\documentclass[11pt,reqno]{amsart}

\usepackage{amscd,amssymb,amsmath,amsthm}
\usepackage{graphicx}
\usepackage{color}
\usepackage{cite}
\newtheorem{thm}{Theorem}
\newtheorem{defn}{Definition}
\newtheorem{lemma}{Lemma}

\newtheorem{rk}{Remark}

\numberwithin{equation}{section} 

\makeatletter
\newcommand*{\rom}[1]{\expandafter\@slowromancap\romannumeral #1@}
\makeatother

\begin{document}
\title[Periodic points of a $p$-adic operator]{Periodic points of a $p$-adic operator and their $p$-adic Gibbs measures}

\author{U.~A.~Rozikov, I.~A.~Sattarov, A.~M.~Tukhtabaev}

\address{U.~A.~Rozikov$^{1,2,3}$, I.~A.~Sattarov$^{1,4}$, A.~M.~Tukhtabaev$^4$}

\address{1. Institute of mathematics,
9, University str. Tashkent, 100174, Uzbekistan,}
\address{2. AKFA University, 264, Milliy Bog street,  Yangiobod QFY, Kibray district, 111221, Tashkent region, Uzbekistan,}
\address{3. National University of Uzbekistan,
 University str 4 Olmazor district, Tashkent, 100174, Uzbekistan.}

\address{4. Namangan State University, 316, Uychi
 str. Namangan, 160100, Uzbekistan.}

\email {rozikovu@yandex.ru, sattarovi-a@yandex.ru, akbarxoja.toxtaboyev@mail.ru}

\begin{abstract} In this paper we investigate generalized Gibbs measure (GGM) for $p$-adic Hard-Core(HC) model with a countable set of spin values on a Cayley tree of order $k\geq 2$. This model is defined
	by $p$-adic parameters $\lambda_i$, $i\in \mathbb N$.
	We analyze  $p$-adic functional equation which provides the consistency condition for the finite-dimensional generalized  Gibbs distributions. Each solutions of the functional equation defines a GGM by $p$-adic version of Kolmogorov's theorem.	
	
	We define $p$-adic Gibbs distributions as limit of the consistent family of finite-dimensional generalized  Gibbs distributions and  show that, for our $p$-adic HC model on a Cayley tree, such a Gibbs distribution does not exist. Under some conditions on parameters $p$, $k$ and $\lambda_i$ we find the number of translation-invariant and two-periodic GGMs for the $p$-adic HC model on the Cayley tree of order two.
\end{abstract}
\maketitle

{\bf Mathematics Subject Classifications (2010).} 82B26 (12J12 46S10 60K35)

{\bf{Key words.}} Cayley tree, Gibbs measure, Hard-Core model, $p$-adic generalized Gibbs measure.

\section{Introduction}

Let $\mathbb Q$ be the set of rational numbers. It is known that all numbers found in nature through rational numbers (by changing definition of metric on $\mathbb Q$) are of only two types:

1) The real numbers (the field of numbers that have Archimedian property\footnote{https://en.wikipedia.org/wiki/Archimedean$_-$property});

2) $p$-adic numbers, where $p$ is a prime number. The   non-Archimedean field of  numbers.

So, the problems of studying nature (related to the rational numbers) are modeled only by these two types of numbers, there are no different, third types of numbers! This is a mathematically proven assertion (Ostrowski's theorem, \cite{Ko}).

The theory of $p$-adic numbers is one of very actively developing area in mathematics. Today, numerous applications of $p$-adic numbers are found in many branches of mathematics,
biology, physics and other sciences (see for example \cite{AK}, \cite{KJ}, \cite{Rpa}, \cite{RS}, \cite{RS2}, \cite{S},  \cite{VVZ1994} and the references therein).

In physics, it is confirmed that it is possible to study very small objects through $p$-adic numbers. For example, in standard quantum mechanics, the wave functions describing the evolution of free particles were expressed as wave functions of $p$-adic strings. This result is explained by the fact that the energy of a simple quantum particle actually consists of the sum of the energies of its $p$-adic components. In this paper we consider a $p$-adic operator related to a physical system with a countable set of spin values and having $p$-adic valued potential. We study periodic (in particular, fixed) points of this operator and give corresponding to them $p$-adic Gibbs measures.

\section{Preliminaries}

\subsection{$p$-adic numbers} Let $\mathbb Q$  be the field of rational numbers. It is clear that, for a fixed prime number $p$,
 every $x\in\mathbb Q$,  $ x\neq 0$ can be represented in the form
$x = {p^r}\frac{n}{m}$, where $r,n \in {\mathbb Z}$, $m$ is a
positive integer, moreover $n$ and $m$ are relatively prime with $p$.
This number $r$ is usually (see \cite{Ka}, \cite{Ko}) denoted by ord$_px$, i.e.,
$$r={\rm ord}_px=\left\{\begin{array}{ll}
	\mbox{the highest power of} \, p \, \mbox{which divides} \, x, \,  \mbox{if} \, x\in \mathbb Z,\\[2mm]
	{\rm ord}_pa-{\rm ord}_pb, \, \mbox{if} \, x=a/b, \, a,b\in \mathbb Z, \, b\ne 0.
	\end{array}
	\right.
$$
The $p$-adic norm of $x\in\mathbb Q$ is given by
$$|x|_p = \left\{ \begin{gathered}
  {p^{- r}},\,x \ne 0, \hfill \\
  0,\,\,\,\,\,x = 0. \hfill \\
\end{gathered}  \right.$$

It is well-known that (see \cite{Ko}, \cite{VVZ1994}) this norm is non-Archimedean, i.e. it satisfies the strong
triangle inequality:
$$
|x + y|_p\leq\max \{ |x|_p,|y|_p\},\ \ \ \ \forall x,y\in\mathbb Q.
$$

From this property one get the following

1) if $|x{|_p} \ne |y{|_p}$, then $|x \pm y|_p =\max\{|x{|_p},|y{|_p}\};$

2) if $|x{|_p} = |y{|_p}$, then $|x - y{|_p}\leq|x{|_p}$.

The completion of $\mathbb Q$ with respect to the $p$-adic norm
defines the $p$-adic field $\mathbb Q_p$.
Any $p$-adic number $x \ne 0$ can be uniquely represented in the
canonical form
$$x = {p^{\gamma (x)}}({x_0} + {x_1}p + {x_2}{p^2} + ...),$$
where $\gamma (x)={\rm ord}_px\in {\mathbb Z}$  and ${x_0}\neq 0, {x_j}\in \{0,1,..., p - 1\}, j=1,2,...$. In this case $|x|_p ={p^{ - \gamma (x)}}$.

In \cite{MKh2020} authors have introduced new symbols ``O" and ``o" which allowed to simplify certain calculations. Roughly speaking, these symbols replace the notation $\equiv \ ({\rm mod} \ p^k)$ without noticing about power of $k$. Let us recall them. For a given $p$-adic number $x$ by $O[x]$ we mean a $p$-adic number such that $|x|_p=|O(x)|_p$. By
 $o[x]$, we mean a $p$-adic number such that $|o(x)|_p<|x|_p$.
 For instance,
 if $x=1-p+p^2$, we can write $O[1]=x$, $o[1]=x-1$ or $o[p]=x-1+p$.
Therefore, the symbols $O[\cdot]$ and $o[\cdot]$ make our work easier when we need
to calculate the $p$-adic norm of $p$-adic numbers.
It is easy to see that $y=O[x]$ if and only if $x=O[y]$.

\subsection{$p$-adic quadratic  equation}
We recall that an integer $a\in {\mathbb Z}$ is called \emph{quadratic residue modulo $p$}
if the congruent equation ${x^2}\equiv{a} (\operatorname{mod }p)$ has a solution $x\in{\mathbb Z}$.

Let $p$ be odd prime and $a$ an integer not divisible by $p$. The \emph{Legendre symbol} (see \cite{Rosen}) is defined by
\begin{equation}\label{SL}
	\left(\frac{a}{p}\right)=\left\{
	\begin{array}{ll}
		1,\ \ \ \ \mbox{if} \ \ a \ \mbox{is quadratic residue of} \ \ p; \\[2mm]
		-1,\ \ \ \mbox{if} \ \ a \ \mbox{is quadratic nonresidue of}\ \ p.
	\end{array}\right.
\end{equation}

\begin{lemma}\label{lemmax2=a}\cite{VVZ1994} The equation $${x^2} = a,\,\, a
= {p^{\gamma (a)}}({a_0} + {a_1}p + {a_2}{p^2} + ...),\,\, 0 \leq
{a_j} \leq p - 1, \, {a_0} > 0$$ has a solution in $Q_p$ iff hold true the following:

i) $ \gamma (a)$ is even;

ii) ${x^2} \equiv {a_0}(\operatorname{mod }p)$ is solvable for $p \ne
2$; the equality ${a_1} = {a_2} = 0$ hold if $p = 2$.
\end{lemma}

\subsection{Cubic equations on $\mathbb Q_p$}
It should be noted that the solvability of the general cubic equation over $\mathbb Q_p$
is equivalent to the solvability of the depressed cubic equation over $\mathbb Q_p$.
Namely, we know that the general cubic equation can be reduced to the following depressed cubic equation
\begin{equation}\label{gen}
	x^3+ax=b,
\end{equation}
here $a,b\in\mathbb Q_p,$ and $ab\neq0.$

In \cite{MOS2014} and \cite{SM2015},  criteria a for solvability of the depressed cubic equation over $\mathbb Q_p$ are given.
We use these criteria in this paper, therefore here we give them.

\textbf{Case $p>3$.}
Let $a=\frac{a^*}{|a|_p}$, $b=\frac{b^*}{|b|_p}$ and $a^*,b^*\in\mathbb Z_p^*=\{x\in \mathbb Q_p: |x|_p=1\}$.

Suppose that $$a^*=a_0+a_1p+a_2p^2+..., \ \ \ b^*=b_0+b_1p+b_2p^2+... .$$
We denote by $D_0=-4a_0^3-27b_0^3,$ $u_1=0, u_2=-a_0, u_3=b_0$, $u_{n+3}=b_0u_n-a_0u_{n+1},$ for any $n=1,2, \dots, p-3$.

\begin{thm}\label{cubeq}\cite{MOS2014} Let $p>3$ be a prime. The equation \eqref{gen} has a solution in $\mathbb Q_p$ if and only if one of the following conditions holds:
	\begin{enumerate}
		\item  $|a|_p^3<|b|_p^2$, $3| {\rm ord}_pb$ and $b_0^{\frac{p-1}{(3,p-1)}}\equiv 1 (\operatorname{mod }p)$;
		\item  $|a|_p^3=|b|_p^2$ and $D_0u_{p-2}^2\not\equiv 9a_0^2(\operatorname{mod }p)$;
		\item  $|a|_p^3>|b|_p^2.$
	\end{enumerate}
\end{thm}
Let $D=-4(a|a|_p)^3-27(b|b|_p)^2\neq0, \ \ D=\frac{D^*}{|D|_p},\ \ D^{*}\in\mathbb Z^{*}_p, \ \ D^*=d_0+d_1p+... .$

\begin{thm}\label{Nsol} \cite {MOS2014}
	Let $p>3$ be a prime and $N$ is number of the solutions of (\ref{gen}) in $\mathbb{Q}_p$. Then the following statements hold true:
	\begin{equation}\label{N}
		N=\left\{
		\begin{array}{lllllllllllll}
			3,\ \ \ \ |a|_p^3<|b|_p^2, 3| {\rm ord}_pb, p\equiv 1({\rm mod} 3), b_0^\frac{p-1}{3}\equiv1 (\operatorname{mod }p); \\[1mm]
			3,\ \ \ \ |a|_p^3=|b|_p^2, D=0;\\[1mm]
			3,\ \ \ \ |a|_p^3=|b|_p^2, 0<\left|D\right|_p<1, 2| {\rm ord}_pD,d_0^\frac{p-1}{2}\equiv1 (\operatorname{mod }p);\\[1mm]
			3,\ \ \ \ |a|_p^3=|b|_p^2, \left|D\right|_p=1 \ \ \mbox {and} \ \ u_{p-2}\equiv 0 (\operatorname{mod }p); \\[1mm]
			3,\ \ \ \ |a|_p^3>|b|_p^2, 2|{\rm ord}_pa, \left(-a_0\right)^\frac{p-1}{2}\equiv 1 (\operatorname{mod }p); \\[1mm]
			1,\ \ \ \ |a|_p^3<|b|_p^2, 3|{\rm ord}_pb, p\equiv 2({\rm mod} 3); \\[1mm]
			1,\ \ \ \ |a|_p^3=|b|_p^2, 0<\left|D\right|_p<1, 2|{\rm ord}_pD,d_0^\frac{p-1}{2}\not\equiv1 (\operatorname{{\rm mod} }p);\\[1mm]
			1,\ \ \ \ |a|_p^3=|b|_p^2, 0<\left|D\right|_p<1, 2\nmid {\rm ord}_pD;\\[1mm]
			1,\ \ \ \ |a|_p^3=|b|_p^2, D_0u_{p-2}^2\not\equiv 0(\operatorname{mod }p), D_0u_{p-2}^2\not\equiv 9a_0^2(\operatorname{mod }p); \\[1mm]
			1,\ \ \ \ |a|_p^3>|b|_p^2, 2|{\rm ord}_pa, \left(-a_0\right)^\frac{p-1}{2}\not\equiv 1 (\operatorname{mod }p); \\[1mm]
			1,\ \ \ \ |a|_p^3>|b|_p^2, 2\nmid {\rm ord}_pa; \\[1mm]
			0,\ \ \ \ otherwise.
		\end{array}\right.
	\end{equation}
\end{thm}

\section{Periodic points of the $p$-adic operator}

Let $\mathbb Q_p$ be the field of $p$-adic numbers and let $$c_0=\left\{ \{x_n\}_{n=1}^{\infty}\subset\mathbb Q_p\mid \lim_{n\to\infty}x_n=0 \right\}.$$
Note that $c_0$ is a linear space over $\mathbb Q_p$. We define an operator $F:x=\{x_n\}_{n=1}^{\infty}\in c_0\to y=Fx=\{y_n\}_{n=1}^{\infty}\in c_0$ by the following formula
\begin{equation}\label{o1}
y_n=\lambda_n\cdot\left({1\over{1+\sum_{n=1}^{\infty}x_n}}\right)^2, \end{equation}
where $\lambda=\{\lambda_n\}_{n=1}^{\infty}\in c_0$ is parameter-vector.

In this section we investigate fixed points of the operator (\ref{o1}).

\subsection{Fixed Points.} First, we classify the set of fixed points ${\rm Fix}(F)$. To do this, we find all solutions of the following equation:
\begin{equation}\label{fix}
x_n=\lambda_n\cdot\left({1\over{1+\sum_{n=1}^{\infty}x_n}}\right)^2, \, \lambda=\{\lambda_n\}_{n=1}^{\infty}, \, x=\{x_n\}_{n=1}^{\infty}\in c_0.
\end{equation}

It is known from $p$-adic analysis that for a series to be convergent, it is necessary and sufficient that the limit of its $n$-term is zero. Since $\lambda=\{\lambda_n\}_{n=1}^{\infty}, \, x=\{x_n\}_{n=1}^{\infty}\in c_0$, we can denote $z=\sum_{n=1}^{\infty}x_n$ and $\theta=\sum_{n=1}^{\infty}\lambda_n$. 

If $z=-1$ then the equation (\ref{fix}) does not make sense. Hence, we assume that $z\neq-1$. Summarizing both hand sides of the equation \eqref{fix} we get 
\begin{equation}\label{fix1}
z(1+z)^2=\theta, \ \ z, \theta\in \mathbb Q_p.
\end{equation}

By substitution $z=t-\frac{2}{3}$ we get
\begin{equation}\label{eqt}
t^3-\frac{t}{3}-\left(\theta+\frac{2}{27}\right) =0.
\end{equation}

This is equation (\ref{gen}) with $a=-\frac{1}{3}, b=\theta+\frac{2}{27}$. For $p>3$ we get $|a|_p=1$. 

We denote
\begin{equation}\label{ab}
-{1\over 3}=a_0+a_1p+a_2p^2+..., \ \ \theta+\frac{2}{27}={1\over{|\theta+\frac{2}{27}|_p}}\cdot(b_0+b_1p+b_2p^2+...),
\end{equation}
$$D_0=-4a_0^3-27b_0^3, \ \ u_1=0, \, u_2=-a_0, \,  u_3=b_0,$$ $$u_{n+3}=b_0u_n-a_0u_{n+1}, \mbox{for any} \, n=1,2, \dots, p-3.$$
Then using Theorem \ref{cubeq} we get the following:

\begin{lemma}\label{existsol} Let $p>3$ be a prime and $\theta\neq-\frac{2}{27}$. The equation \eqref{eqt} has a solution in $\mathbb Q_p$ if and only if one of the following conditions holds true:
\begin{enumerate}
 \item $|\theta+\frac{2}{27}|_p>1$, $3|{\rm ord}_p \left(\theta+\frac{2}{27}\right)$ and $b_0^{\frac{p-1}{(3,p-1)}}\equiv 1 (\operatorname{mod }p)$;
 \item $|\theta+\frac{2}{27}|_p=1$ and $D_0u_{p-2}^2\not\equiv 9a_0^2(\operatorname{mod }p)$;
 \item $|\theta+\frac{2}{27}|_p<1$.
 \end{enumerate}
\end{lemma}
\begin{rk}\label{ab=0} If $\theta=-\frac{2}{27}$ then the equation \eqref{eqt} has at least one solution in $\mathbb Q_p$, i.e. $t=0$ is solution of the equation \eqref{eqt}.
\end{rk}

Let $D={4\over 27}-27(\theta+\frac{2}{27})^2|\theta+\frac{2}{27}|_p^2$. If $D\neq 0$, then we denote $D=\frac{D^*}{|D|_p},  D^{*}\in\mathbb Z^{*}_p,$ $D^*=d_0+d_1p+... \, .$
Then using Theorem \ref{Nsol} we get the following theorem about the classification of the set ${\rm Fix}(F)$.

\begin{thm}\label{N1} Let $p>3$ be a prime and $|{\rm Fix}(F)|$ is number of the fixed points of the operator $F$. Then the following statements hold:
$$
|{\rm Fix}(F)|=\left\{
\begin{array}{llllllllllll}
3,\ \ \ \ |\theta+\frac{2}{27}|_p^2>1, 3|{\rm ord}_p(\theta+\frac{2}{27}), p\equiv 1(\operatorname{mod }3), b_0^\frac{p-1}{3}\equiv1 (\operatorname{mod }p); \\[1mm]
3,\ \ \ \ |\theta+\frac{2}{27}|_p=1, D=0;\\[1mm]
3,\ \ \ \ |\theta+\frac{2}{27}|_p=1, 0<\left|D\right|_p<1, 2|{\rm ord}_pD,d_0^\frac{p-1}{2}\equiv1 (\operatorname{mod }p);\\[1mm]
3,\ \ \ \ |\theta+\frac{2}{27}|_p=1, \left|D\right|_p=1 \ \ 
\mbox{and} \ \ u_{p-2}\equiv 0 (\operatorname{mod }p); \\[1mm]
3,\ \ \ \ |\theta+\frac{2}{27}|_p<1, \left(-a_0\right)^\frac{p-1}{2}\equiv 1 (\operatorname{mod }p); \\[1mm]
1,\ \ \ \ |\theta+\frac{2}{27}|_p>1, 3|{\rm ord}_p(\theta+\frac{2}{27}), p\equiv 2(mod 3); \\[1mm]
1,\ \ \ \ |\theta+\frac{2}{27}|_p=1, 0<\left|D\right|_p<1, 2|{\rm ord}_pD,d_0^\frac{p-1}{2}\not\equiv1 (\operatorname{mod }p);\\[1mm]
1,\ \ \ \ |\theta+\frac{2}{27}|_p=1, 0<\left|D\right|_p<1, 2\nmid {\rm ord}_pD;\\[1mm]
1,\ \ \ \ |\theta+\frac{2}{27}|_p=1, D_0u_{p-2}^2\not\equiv 0(\operatorname{mod }p), D_0u_{p-2}^2\not\equiv 9a_0^2(\operatorname{mod }p); \\[1mm]
1,\ \ \ \ |\theta+\frac{2}{27}|_p<1, \left(-a_0\right)^\frac{p-1}{2}\not\equiv 1 (\operatorname{mod }p); \\[1mm]
0,\ \ \ \ otherwise.
\end{array}\right.
$$
\end{thm}
\begin{proof} To prove this theorem, we have to check all  conditions of Theorem \ref{Nsol}. However, we shall content ourselves with checking the first condition, since the rest can be checked similarly. First condition of Theorem \ref{Nsol} has the following form
 $$|a|_p^3<|b|_p^2, \, 3| {\rm ord}_pb, \, p\equiv 1(\operatorname{mod }3), \, b_0^\frac{p-1}{3}\equiv1 (\operatorname{mod }p).$$ In our case $a=-{1\over 3}$ and $b=\theta+\frac{2}{27}$. Since $p>3$, by (\ref{ab}) we have $$\left|\theta+\frac{2}{27}\right|_p^2>1, 3|{\rm ord}_p\left(\theta+\frac{2}{27}\right), p\equiv 1(\operatorname{mod }3), b_0^\frac{p-1}{3}\equiv1 (\operatorname{mod }p).$$
\end{proof}

\textbf{Case $p=3$.}
In this case we denote
$$\mathbb Z_3^*[i_0,i_1,...,i_k]=\left\{x^*\in\mathbb Z_3^*:x^*=i_0+i_1 3+...+i_k 3^k+x_{k+1} 3^{k+1}+...\right\},$$
$$\mathbb Z_3^*[i_0,i_1,...,i_k|j_0,j_1,...,j_s]=\mathbb Z_3^*[i_0,i_1,...,i_k]\times \mathbb Z_3^*[j_0,j_1,...,j_s],$$
$${\Delta_{11}} = \bigcup\limits_{i,j=0}^2 {\mathbb Z_3^*} [2,i,j|1,2,i,j], \ \  {\Delta_{12}} =\bigcup\limits_{i,j=0}^2 {\mathbb Z_3^*} [2,1,j|1,2,1,j+1],$$
$${\Delta_{13}} = \bigcup\limits_{i,j=0}^2 {\mathbb Z_3^*} [2,i+1,j+1|1,2,i+1,j],$$ $${\Delta_{21}} =\bigcup\limits_{i,j=0}^2 {\mathbb Z_3^*} [2,i+j,i|1,0,2-(i+j),j], \ \ {\Delta_{22}} = \bigcup\limits_{i,j=0}^2 {\mathbb Z_3^*} [2,0,2-j|2,0,2,j],$$ $${\Delta_{23}} = \bigcup\limits_{i,j=0}^2 {\mathbb Z_3^*} [2,3+i,j|2,0,i-1,1-(i+j)],$$
$$\Delta_1=\Delta_{11}\cup\Delta_{12}\cup\Delta_{13}, \ \  \Delta_2=\Delta_{21}\cup\Delta_{22}\cup\Delta_{23}, \ \ \Delta=\Delta_1\cup\Delta_2.$$
\begin{thm}\label{exstp3}\cite{SM2015} The equation \eqref{gen} has a solution in $\mathbb Q_3$ if and only if one of the following conditions holds true:
\begin{enumerate}
  \item $\left|a\right|_3^3>\left|b\right|_3^2$;
  \item $\left|a\right|_3^3=\left|b\right|_3^2$ and $a^*\in\mathbb Z_3^*[1]$;
  \item $\left|a\right|_3^3<\left|b\right|_3^2$, $3| {\rm ord}_3b$ and 
  
i) $\left|\frac{a}{3}\right|_3^3=\left|b\right|_3^2$, $(a^*,b^*)\in\mathbb Z_3^*[1|1,1]\cup\mathbb Z_3^*[1|2,1]\cup\Delta$;

ii) $\left|\frac{a}{3}\right|_3^3<\left|b\right|_3^2$, $b^*\in\mathbb Z_3^*[1,0]\cup\mathbb Z_3^*[2,2]$.
\end{enumerate}
\end{thm}

\begin{thm}\label{Nsolp3}\cite{SM2015} Let $N_2$ denote number of the solutions of (\ref{gen}) in $\mathbb{Q}_3$. Then the following statements hold true:
\begin{equation}\label{N2}
N_2=\left\{
\begin{array}{lllllllll}
3,\ \ \ \ |a|_3^3>|b|_3^2, 2| {\rm ord}_3a, a^*\in\mathbb Z_3^*[2]; \\[1mm]
1,\ \ \ \ |a|_3^3>|b|_3^2, 2| {\rm ord}_3a, a^*\not\in\mathbb Z_3^*[2];\\[1mm]
1,\ \ \ \ |a|_3^3>|b|_3^2, 2\nmid {\rm ord}_3a;\\[1mm]
1,\ \ \ \ |a|_3^3=|b|_3^2, a^*\in\mathbb Z_3^*[1]; \\[1mm]
1,\ \ \ \ |a|_3^3<|b|_3^2, 3| {\rm ord}_3b, \left|\frac{a}{3}\right|_3^3=|b|_3^2, \\[1mm]
   \ \ \ \ (a^*,b^*)\in\mathbb Z_3^*[1|2,1]\cup\mathbb Z_3^*[1|1,1]\cup\Delta; \\[1mm]
1,\ \ \ \ |a|_3^3<|b|_3^2, 3| {\rm ord}_3b, \left|\frac{a}{3}\right|_3^3<|b|_3^2, \\[1mm]
     \ \ \ \ b^*\in\mathbb Z_3^*[1,0]\cup\mathbb Z_3^*[2,2].\\[1mm]
\end{array}\right.
\end{equation}
\end{thm}

In our case $a=-\frac{1}{3}=3^{-1}(2+2\cdot3+2\cdot3^2+...), \, b=\theta+\frac{2}{27}$ and $|a|_3=3.$ From this we have ${\rm ord}_3a=-1$,
 $a^*=a|a|_p=2+2\cdot3+2\cdot3^2+...$ and $a^*\in\mathbb Z_3^*[2,2,...,2]$.

 Using Theorem \ref{exstp3} we get the following

\begin{lemma}\label{exsolp3} The equation \eqref{eqt} has a solution in $\mathbb Q_3$ if and only if one of the following conditions holds true:
\begin{itemize}
  \item $|\theta+\frac{2}{27}|_3^2<27$;
  \item $\left|\theta+\frac{2}{27}\right|_3>27$, $3|{\rm ord}_3\left(\theta+\frac{2}{27}\right)$, $b^*\in\mathbb Z_3^*[1,0]\cup\mathbb Z_3^*[2,2]$.
\end{itemize}
\end{lemma}
\begin{proof} We are going to check all conditions of  Theorem \ref{exstp3} for our case.\\
1)$\left|a\right|_3^3>\left|b\right|_3^2$ yields $\left|\theta+\frac{2}{27}\right|_3^2<27.$ Hence, the
first condition of Theorem \ref{exstp3} is satisfied.\\
2) $\left|a\right|_3^3=\left|b\right|_3^2$ and $a^*\in\mathbb Z_3^*[1]$ yield $\left|\theta+\frac{2}{27}\right|_3^2=27$, but in our case $a^*\in\mathbb Z_3^*[2]$. Hence, the second condition of Theorem \ref{exstp3} does not hold in our case.\\
3) $\left|a\right|_3^3<\left|b\right|_3^2$, $3|{\rm ord}_3 b$ yield $\left|\theta+\frac{2}{27}\right|_3^2>27$, $3| {\rm ord}_3\left(\theta+\frac{2}{27}\right)$ respectively.\\
i) $\left|\frac{a}{3}\right|_3^3=\left|b\right|_3^2$, $(a^*,b^*)\in\mathbb Z_3^*[1|1,1]\cup\mathbb Z_3^*[1|2,1]\cup\Delta$ yield $|\theta+\frac{2}{27}|_3=27$,$(a^*,b^*)\in\mathbb Z_3^*[1|1,1]\cup\mathbb Z_3^*[1|2,1]\cup\Delta$. But $(a^*,b^*)\not\in\mathbb Z_3^*[1|1,1]\cup\mathbb Z_3^*[1|2,1]\cup\Delta$. Hence, this condition of Theorem \ref{exstp3} does not hold in our case.\\
ii) $\left|\frac{a}{3}\right|_3^3<\left|b\right|_3^2$, $b^*\in\mathbb Z_3^*[1,0]\cup\mathbb Z_3^*[2,2]$ yield $\left|\theta+\frac{2}{27}\right|_3>27$, $b^*\in\mathbb Z_3^*[1,0]\cup\mathbb Z_3^*[2,2]$. Hence, this condition of Theorem \ref{exstp3} is satisfied.
\end{proof}

\begin{thm}\label{unique} Let $p=3$. If one of the conditions of Lemma \ref{exsolp3} is satisfied, then the operator $F$ has a unique fixed point.

Otherwise, there is no fixed point of the operator $F$.
\end{thm}
\begin{proof} Note that the number of fixed points of the operator $F$ is equal to the number of solutions of the equation (\ref{eqt}).
We use Theorem \ref{Nsolp3} to find the number of solutions of the equation (\ref{eqt}).

First, we have all the conditions for the existence of a fixed point of the operator $F$ from Lemma \ref{exsolp3}.
Theorem \ref{Nsolp3} considers six cases, for the operator $F$ there are only two of these cases, i.e.,
only in the cases $|a|_3^3>|b|_3^2, \, 2\nmid {\rm ord}_3a$ and $|a|_3^3<|b|_3^2, \,  3| {\rm ord}_3b,$ $\left|\frac{a}{3}\right|_3^3<|b|_3^2, \ \ b^*\in\mathbb Z_3^*[1,0]\cup\mathbb Z_3^*[2,2]$ there is a fixed point of the operator $F$. According to Theorem \ref{Nsolp3},
the solution in these conditions is unique. It follows that if the operator $F$ has a fixed point, it will be unique.
\end{proof}

\subsection{$2$-Periodic Points.}\label{s32}
We now deal with the classification of the set of $2$-periodic points of the operator $F$.
Denote 
$${\rm Per}_2\{F\}=\{x\in c_0| \ \ FFx=x\}\setminus {\rm Fix}\{F\}.$$

By (\ref{o1}), the equation $FFx=x$ has the following form

\begin{equation}\label{subs}
x_n=\frac{\lambda_n}{\left(1+\sum_{j\in \mathbb N} \frac{\lambda_j}{\left(1+\sum_{i\in \mathbb N} x_i\right)^2}\right)^2}, \, \lambda=\{\lambda_n\}_{n=1}^{\infty}, \, x=\{x_n\}_{n=1}^{\infty}\in c_0.
\end{equation}

Let $\theta=\sum_{j\in \mathbb N} \lambda_i$ and $z=\sum_{j\in \mathbb N} x_j$.
Summarizing both hand sides of \eqref{subs} we get
\begin{equation}\label{eqper}
z=\frac{\theta \left(1+z\right)^{4}}{\left(\left(1+z\right)^2+\theta \right)^2}.
\end{equation}

Let $f(z)=\frac{\theta}{\left(1+z\right)^2}$ then $f\left(f(z)\right)=\frac{\theta \left(1+z\right)^{4}}{\left(\left(1+z\right)^2+\theta \right)^2}.$
In order to find $2$-periodic points (different from the fixed points) of $f(z)$ we consider the following equation:
\begin{equation}\label{eqper2}
\frac{f\left(f(z)\right)-z}{f\left(z\right)-z}=0.
\end{equation}

From the \eqref{eqper2} we get
\begin{equation}\label{eqperk=2}
z^2+(2-\theta)z+1=0.
\end{equation}
The solutions of the equation \eqref{eqperk=2} are
\begin{equation}\label{solperk=2}
z_{1,2}=\frac{\theta-2\pm\sqrt{\theta^2-4\theta}}{2}.
\end{equation}

Substituting \eqref{solperk=2} into \eqref{fix1} we have
\begin{equation}\label{pernottrinv}
\theta \sqrt{\theta^2-4\theta}\left(\sqrt{\theta^2-4\theta}\pm (\theta-2)\right)=0.
\end{equation}

It is clear that if $\theta\neq0$ and $\theta\neq4$ then the solutions of the equation \eqref{eqperk=2} are not fixed points of $f(z)$.

If $\theta=0$ the solution \eqref{solperk=2} gives $z=-1$, however, we assumed that $z\neq-1$ (see \eqref{fix}).

If $\theta=4$ the solution \eqref{solperk=2} gives $z=1$. In this case $z=1$ is fixed point of $f(z)$.

Thus, we can assume that $\theta\neq0$ and $\theta\neq4$.
Let $\theta=p^{\gamma(\theta)}(\theta_0+\theta_1p+\theta_2p^2+...), \, \theta_0\neq0$ and $D(\theta)=\theta^2-4\theta$.

\begin{thm}\label{NPer2} Let $p\geq3$ be a prime and $|{\rm Per}_2\{F\}|$ be the number of $2$-periodic points of the operator $F$. Then the following statements hold:
\begin{equation}\label{no}
|{\rm Per}_2\{F\}|=\left\{
\begin{array}{lllllll}
2, \mbox ~{if}~ \left|\theta\right|_p>1;\\[2mm]
2, \mbox ~{if}~ \left|\theta\right|_p<1, 2\mid {\rm ord}_p \theta, \left(\frac{-\theta_0}{p}\right)=1;\\[2mm]
0, \mbox ~{if}~ \left|\theta\right|_p<1, 2\nmid {\rm ord}_p \theta \mbox ~{or}~\left(\frac{-\theta_0}{p}\right)=-1;\\[2mm]
2, \mbox ~{if}~ \left|\theta\right|_p=1, \left(\frac{\theta_0^2-4\theta_0}{p}\right)=1;\\[2mm]
0, \mbox ~{if}~ \left|\theta\right|_p=1, \left(\frac{\theta_0^2-4\theta_0}{p}\right)=-1;\\[2mm]
2, \mbox ~{if}~ \left|\theta\right|_p=1, 2\mid {\rm ord}_p (\theta-4), \left(\frac{|\theta-4|_p(\theta-4)(\operatorname{mod }p)}{p}\right)=1;\\[2mm]
0, \mbox ~{if}~ \left|\theta\right|_p=1, 2\nmid {\rm ord}_p (\theta-4)  \mbox ~{or}~ \left(\frac{|\theta-4|_p(\theta-4)(\operatorname{mod }p)}{p}\right)=-1;\\[2mm]
\end{array}\right.
\end{equation}
\end{thm}
\begin{proof}
Let $\left|\theta\right|_p>1$, i.e. $\gamma(\theta)<0.$  Then we rewrite $D(\theta)=p^{2\gamma(\theta)}\left(\theta_0^2+o[1]\right).$
Due to the Lemma \ref{lemmax2=a} the equation \eqref{eqperk=2} has two solutions.

Let $\left|\theta\right|_p<1$, i.e. $\gamma(\theta)>0.$ Then we can rewrite $D(\theta)=p^{\gamma(\theta)}\left(-4\theta_0+o[1]\right).$
Then, again thanks to the Lemma \ref{lemmax2=a} the equation \eqref{eqperk=2} has two solutions iff $\gamma(\theta)$ is even (it means that $2\mid {\rm ord}_p \theta$) and the congruence $x^2\equiv-\theta_0 (\operatorname{mod }p)$ is solvable (It means that $\left(\frac{-\theta_0}{p}\right)=1$).

Let $\left|\theta\right|_p=1$, i.e. $\gamma(\theta)<0.$ Then $D(\theta)=\theta_0^2-4\theta_0+o[1]$, here $\theta_0\neq0$.
So, we have the following two cases:

Let $\theta_0\neq4$. It follows that the equation \eqref{eqperk=2} has two solutions iff the congruence $x^2\equiv\theta_0^2-4\theta_0 (\operatorname{mod }p)$ is solvable.

Let $\theta_0=4$. Since $\theta\neq4$, then there exists $s\in \mathbb N$ such that $\theta=4+\theta_sp^s+o[p^s]$. It yields that $D(\theta)=p^s\left(4\theta_s+o[1]\right)$. According to the Lemma \ref{lemmax2=a}, the equation \eqref{eqperk=2} has two solutions if and only if $s$ is even and the congruence $x^2\equiv\theta_s (\operatorname{mod }p)$ is solvable. It means that $\left(\frac{|\theta-4|_p(\theta-4)(\operatorname{mod }p)}{p}\right)=1$.
\end{proof}

\begin{thm}\label{Tp2} Let $p=2$ and $|{\rm Per}_2(F)|$ is number of $2$-periodic points of the operator $F$. Then the following statements hold:
\begin{equation}\label{no2}
|{\rm Per}_2(F)|=\left\{
\begin{array}{llllllllllllllll}
2, \mbox ~{if}~ |\theta|_2>1;\\[2mm]
2, \mbox ~{if}~ |\theta|_2=1 \mbox~{and}~ \theta_2=1;\\[2mm]
0, \mbox ~{if}~ |\theta|_2=1 \mbox~{and}~ \theta_2=0;\\[2mm]
0, \mbox ~{if}~ |\theta|_2\leq\frac{1}{32}, 2\nmid {\rm ord}_p \theta;\\[2mm]
2, \mbox ~{if}~ |\theta|_2\leq\frac{1}{32},2\mid {\rm ord}_p \theta, \theta_1=\theta_2=0;\\[2mm]
0, \mbox ~{if}~ |\theta|_2\leq\frac{1}{32},2\mid {\rm ord}_p \theta, \theta_1=1 \mbox ~{or}~\theta_2=1;\\[2mm]
2, \mbox ~{if}~ |\theta|_2=\frac{1}{2} \mbox~{and}~\theta_1=0;\\[2mm]
0, \mbox ~{if}~ |\theta|_2=\frac{1}{2} \mbox~{and}~\theta_1=1;\\[2mm]
0, \mbox ~{if}~ |\theta|_2=\frac{1}{4} \mbox~{and}~ \theta_1=1;\\[2mm]
2, \mbox ~{if}~ |\theta|_2=\frac{1}{4} \mbox~{and}~ \theta_1=0,\theta_2=\theta_3=1,\theta_4=0;\\[2mm]
0, \mbox ~{if}~ |\theta|_2=\frac{1}{4} \mbox~{and}~ (\theta_1-1)\theta_2\theta_3(\theta_4-1)=0;\\[2mm]
0, \mbox ~{if}~ |\theta|_2=\frac{1}{8};\\[2mm]
2, \mbox ~{if}~ |\theta|_2=\frac{1}{16} \mbox~{and}~ \theta_1=0,\theta_2=1;\\[2mm]
0, \mbox ~{if}~ |\theta|_2=\frac{1}{16} \mbox~{and}~ \theta_1=1 \mbox~or~\theta_2=0.
\end{array}\right.
\end{equation}
\end{thm}

\begin{proof}
 Let $|\theta|_2>1.$ Then $D(\theta)=\theta^2-4\theta=\theta^2(1-\frac{4}{\theta})$. We have $\left|\frac{4}{\theta}\right|_2\leq\frac{1}{8}$. Thus, by Lemma \ref{lemmax2=a} the equation \eqref{eqperk=2} has two solutions.

Let $|\theta|_2=1.$ We have $\gamma(\theta)=0$. Then $D(\theta)=1+(\theta_1^2+\theta_1+\theta_2-1)2^2+o[4]$. $\theta_1^2+\theta_1\equiv0(\operatorname{mod }2)$. Thanks to Lemma \ref{lemmax2=a} the equation \eqref{eqperk=2} has two solutions if $\theta_2=1$, the equation \eqref{eqperk=2} has no solution if $\theta_2=0$.

Let $|\theta|_2\leq\frac{1}{32}$. It means $\gamma(\theta)\geq5$. Then $D(\theta)=2^{\gamma(\theta)+2}\left(-1-\theta_12-\theta_22^2-\theta_32^3+o[8]\right)$. By  Lemma \ref{lemmax2=a} the equation \eqref{eqperk=2} has two solutions iff $\theta_1=\theta_2=0$ and $\gamma(\theta)$ is even.

Let $|\theta|_2=\frac{1}{2}$, i.e. $\gamma(\theta)=1$. Then $D(\theta)=2^{2}\left(-1-\theta_1^22^2+\theta_22^4+o[2^4]\right)$.By Lemma \ref{lemmax2=a} the equation \eqref{eqperk=2} has two solutions iff $\theta_1=0.$

Let $|\theta|_2=\frac{1}{4}$, i.e. $\gamma(\theta)=2$. Then \\ $D(\theta)=2^{4}\left(-\theta_12+(\theta_1^2+\theta_1-\theta_2)2^2+(\theta_2-\theta_3)2^3+(\theta_3+\theta_1\theta_2+\theta_2^2-\theta_4)2^4+o[2^4]\right)$. The equation \eqref{eqperk=2} has two solutions iff $\theta_1=0,\theta_2=1,\theta_3=1,\theta_4=0.$

Let $|\theta|_2=\frac{1}{8}$, i.e. $\gamma(\theta)=3$. Then $D(\theta)=2^{5}\left(1-\theta_12-\theta_22^2+o[2^2]\right)$.

 Again by Lemma \ref{lemmax2=a} the equation \eqref{eqperk=2} has no solution.

Let $|\theta|_2=\frac{1}{16}$, i.e. $\gamma(\theta)=4$. Then $D(\theta)=2^{6}\left(-1-\theta_12+(1-\theta_2)2^2+o[2^2]\right)$. The equation \eqref{eqperk=2} has two solutions iff $\theta_1=0,\theta_2=1.$
\end{proof}

\section{Application to the theory of Gibbs measures}

In this section we give $p$-adic Gibbs measures corresponding to fixed and periodic points of operator $F$.

\subsection{Definitions and equations}

The Cayley tree $\Im^{k}=(V, L)$ of order $k \geq 1$ is an infinite tree, i.e. graph without cycles, each vertex of which has exactly $k+1$ edges (see Chapter 1 in \cite{5}). Here $V$ is the set of vertices of $\Im^{k}$ va $L$ is the set of its edges. For $l \in L$ its endpoints $x, y \in V$ are called nearest neighbor vertices and denoted by $l=\langle x, y\rangle$.

The distance  $d(x, y)$ between the vertices $x$ and $y$ on the Cayley tree, is number of edges of the shortest path connecting vertices $x$ and $y$.

For a fixed $x^{0} \in V$ we put
$$
W_{n}=\left\{x \in V \mid d\left(x, x^{0}\right)=n\right\}, \quad V_{n}=\bigcup_{m=0} ^{n} W_{m}. $$

 The set $S(x)$ of direct successors of the vertex $x$ is defined as follows. If $x \in W_{n}$ then
$$
S(x)=\left\{y_{i} \in W_{n+1} \mid d\left(x, y_{i}\right)=1, i=1,2, \ldots, k \right\}.
$$

 We consider the hard core (HC) model of nearest neighbors with a countable number of states from $\mathbb N_0=\{0,1,2,\dots\}$ on the Cayley tree. The configuration $\sigma=\{\sigma(x) \mid x \in V\}$ on the Cayley tree is given as a function from $V$ to the set $\mathbb {N}_0$. For a subset $A\subset \mathbb N_0$ we denote by $\Omega_A$ the set of all configurations defined on $A$.

The activity set (see \cite{16}) on the set of states $\mathbb N_0$ is a bounded function $\lambda: \mathbb N_0 \mapsto \mathbb {N}_0$. 

Now we give definition of $p$-Adic Gibbs measure.
Let $\textbf {z}:x\longrightarrow z_x=(z_{1,x}, z_{2,x}, ...)\in \mathbb Q_p ^{\mathbb N_0}$ be vector-valued function on $V$, where $z_{i,x}\in\mathbb Q_p$, $i\in \mathbb N_0$.

Let $\nu=\{\nu(i)\in\mathbb Q_p, i\in\mathbb N_0\}$ be a fixed $p$-adic probability measure \cite{Kh}. We construct \emph{generalized probability Gibbs distribution} $\mu^{(n)}$ (as a HC-model) on $\Omega_{V_n}$ for any $n\in\mathbb N$ and $\lambda_{\sigma(x)}$ as the following
\begin{equation}\label{mu}
\mu^{(n)}(\sigma_n)=
\frac{1}{Z_n}\prod_{x\in V_n}\lambda_{\sigma_n(x)}\nu(\sigma_n(x)) \prod_{x\in W_n}\textbf {z}_{\sigma_n(x),x},\
\end{equation}
Here $Z_{n}$ is
the corresponding normalizing factor or {\it partition
function} given by
\begin{equation}\label{ZN1}
Z_{n}=\sum_{\sigma_n\in\Omega_{V_n}}\prod_{x\in V_n}\lambda_{\sigma_n(x)}\nu(\sigma_n(x)) \prod_{x\in W_n}\textbf {z}_{\sigma_n(x),x}.
\end{equation}

\begin{defn}\label{pGM} Generalized probability Gibbs distribution is called $p$-adic \emph{Gibbs distribution} if (see \cite{GRR}, \cite{MO})
$$\lambda_i, \textbf {z}_{i,x}\in\mathcal E_p:=\left\{a\in\mathbb Q_p: |a-1|_p<p^{-1/(p-1)}\right\}, \ \ \mbox{for any} \ \ i\in\mathbb N_0, \, x\in V.$$
\end{defn}

By $p$-adic version of Kolmogorov's theorem on extension of measures (see \cite{GMR}) one can show that if measures (\ref{mu}) are consistent (compatible) then there is unique $p$-adic Gibbs measure, which coincides with $\mu^{(n)}$ on the cylinders with base $\sigma_n$. 

  A configuration $\sigma$ defined on $V$ is called \textit{admissible} if $\sigma (x) \sigma (y) = 0 $ for any neighbor $ \langle x, y \rangle $ in $ V $, i.e. if the vertex $ x $ has a spin value $ \sigma (x) = 0 $, then on the neighbor vertices we can put any value from $ \mathbb {N}_0 $, if on the vertex $ x $ there is any value from $ \mathbb{ N}_0 $, then we put only zeros on the neighbor vertices.

Here we consider $p$-adic Gibbs measures on the set of admissible configurations. By similar arguments of the real case (see proof of Theorem 2.1 in \cite{BR}) one can show that 
the measure $\mu^{(n)}, \ n=1,2,...$ associated with \eqref{mu}, on the set of admissible configurations, satisfy compatibility condition if and only if for any $x\in V\setminus\{x^0\}$ the following equation holds:
\begin{equation}\label{comp1} \textbf {z}_{i,x}=\lambda_i \prod_{y \in S(x)} \frac{1}{1+\sum_{j\in \mathbb {N}_0} \textbf {z}_{j,y}}, \ \ {i}\in \mathbb {N}_0.
\end{equation}
here $\lambda_i$ is redefined as function of previous $\lambda_i$ and $\nu(i)$. Assume that $1+\sum_{j\in \mathbb {N}_0} \textbf {z}_{j,y}\neq 0$ for any $y\in S(x)$.

Finding the general form of solutions to the equation (\ref{comp1}) seems to be a very difficult task.

Let we consider following series $\sum_{j\in \mathbb N_0} \lambda_j$ and  $\sum_{j\in \mathbb N_0} \textbf{z}_{j,y},$ $y\in V$. If these series are not convergent, then the above equation does not make sense.

Recall that a $p$-adic series converges if and only if its terms have zero limit. Also, if the $p$-adic series is convergent, its sum does not depend on the numbering of terms (see page 75 of \cite{Ka}).

\begin{thm}\label{existence}
	There does not exist a $p$-adic Gibbs distribution on the set of admissible configurations (with finite-volume distributions defined as \eqref{mu}).
\end{thm}
\begin{proof}
	It can be seen that there is no any solution  $\textbf {z}_{i,x}\in\mathcal E_p$  to the equation (\ref{comp1}). Indeed, assume a solution $\textbf {z}_{i,x}\in\mathcal E_p$ exists. Then by above mentioned remark we should have 
	$$\lim_{i\to \infty}\textbf {z}_{i,x}=0, \ \ \mbox{for all} \ \ x\in V.$$
	That is for any $\varepsilon>0$ there is $i_0\in \mathbb N$ such that 
	$$|\textbf {z}_{i,x}|_p<\varepsilon, \ \ i\geq i_0.$$
	Take now $\varepsilon<1$ then by the property 1) of the $p$-adic norm (see Section 2.1) we get
	$$|\textbf {z}_{i,x}-1|_p=1>p^{-1/(p-1)},$$ that is  a contradiction to the assumption that $\textbf {z}_{i,x}\in\mathcal E_p$.
\end{proof}
\begin{rk}\label{rkmr}
	In \cite{KMR} for the real-valued case of the HC-model it is proven that there is unique TIGM. But Theorem \ref{existence} shows that $p$-adic version of the model is different from the real one. 
\end{rk}
Now, we investigate generalized Gibbs measures (GGMs). 
 Consider translation-invariant solutions, i.e., $\textbf{z}_x=\textbf{z}\in \mathbb Q_p^{\mathbb N_0}$, for all $x\in V$.
GGM corresponding to such a solution, is called translation-invariant $p$-adic GGM (TIGpGM) 

Similarly, one can define a periodic GGM (PGpGM) which corresponds to the periodic solution $\textbf{z}_x$, in sense of periodicity with respect to shifts on vertices of the Cayley tree (see, for example, page 78 of \cite{R99} and page 931 of \cite{Rt}). Moreover, as it was shown in \cite{RH}, only  translation-invariant and two-period Gibbs measures may exist for our HC-model.

\subsection{Translation-invariant solutions for $p\geq 3$.}

To classify the TIGpGMs it is necessary to find all translation-invariant solutions of \eqref{comp1},
i.e., $\textbf{z}_x=\textbf{z}\in \mathbb Q_p^{\mathbb N_0}.$

To do this, we rename the series terms in the equation and get \begin{equation}\label{comp.2}
\textbf{z}_i=\lambda_i\cdot \frac{1}{\left(1+\sum_{j\in \mathbb N} \textbf{z}_j\right)^k}, \ \ i\in \mathbb N.
\end{equation}

It can be seen that the equation (\ref{comp.2}) is equivalent to the equation (\ref{fix}) in the previous section.
Hence, the number of TIGpGMs for the HC model on the Cayley tree of order two
is the same as the number of fixed points of the operator $F$ considered in the previous section.
To use these fixed points we introduce, for $p>3$,  $$\theta=\sum_{n=1}^{\infty}\lambda_n, \ \ b=\theta+\frac{2}{27},$$ 
then 
$$-{1\over 3}=a_0+a_1p+a_2p^2+..., \ \ b|b|_p=b_0+b_1p+...,$$
and 
$$D_0=-4a_0^3-27b_0^3, \ \ u_1=0, u_2=-a_0, u_3=b_0,$$ $$u_{n+3}=b_0u_n-a_0u_{n+1}, \ \ \mbox{for any} \ \ n=1, 2, \dots, p-3.$$

Also, $D={4\over 27}-27b^2|b|_p^2$. If $D\neq 0$, then  $D=\frac{D^*}{|D|_p},  D^{*}\in\mathbb Z^{*}_p,$ $D^*=d_0+d_1p+... \, .$

Now as a corollary of Theorem \ref{N1} we have
\begin{thm}\label{TR1} Let $p>3$ be a prime and $N_{TIGpGM}$ is number of the translation-invariant generalized $p$-adic Gibbs measures for HC model on the Cayley tree of order two. Then the following statements hold:
\begin{equation}\label{tr1}
N_{TIGpGM}=\left\{
\begin{array}{llllllllllll}
3,\ \ \ \ |b|_p^2>1, 3|{\rm ord}_pb, p\equiv 1(\operatorname{mod }3), b_0^\frac{p-1}{3}\equiv1 (\operatorname{mod }p); \\[1mm]
3,\ \ \ \ |b|_p=1, D=0;\\[1mm]
3,\ \ \ \ |b|_p=1, 0<\left|D\right|_p<1, 2|{\rm ord}_pD,d_0^\frac{p-1}{2}\equiv1 (\operatorname{mod }p);\\[1mm]
3,\ \ \ \ |b|_p=1, \left|D\right|_p=1 \ \ \mbox {and} \ \ u_{p-2}\equiv 0 (\operatorname{mod }p); \\[1mm]
3,\ \ \ \ |b|_p<1, \left(-a_0\right)^\frac{p-1}{2}\equiv 1 (\operatorname{mod }p); \\[1mm]
1,\ \ \ \ |b|_p>1, 3|{\rm ord}_pb, p\equiv 2(mod 3); \\[1mm]
1,\ \ \ \ |b|_p=1, 0<\left|D\right|_p<1, 2|{\rm ord}_pD,d_0^\frac{p-1}{2}\not\equiv1 (\operatorname{mod }p);\\[1mm]
1,\ \ \ \ |b|_p=1, 0<\left|D\right|_p<1, 2\nmid {\rm ord}_pD;\\[1mm]
1,\ \ \ \ |b|_p=1, D_0u_{p-2}^2\not\equiv 0(\operatorname{mod }p), D_0u_{p-2}^2\not\equiv 9a_0^2(\operatorname{mod }p); \\[1mm]
1,\ \ \ \ |b|_p<1, \left(-a_0\right)^\frac{p-1}{2}\not\equiv 1 (\operatorname{mod }p); \\[1mm]
0,\ \ \ \ otherwise.
\end{array}\right.
\end{equation}
\end{thm}
\begin{rk}
Comparing with Remark \ref{rkmr} one notes that Theorem \ref{TR1} shows regions of parameters where TIpGGM is non unique. That is another interesting difference between real and $p$-adic cases.  
\end{rk}
In the case $p=3$, by Lemma \ref{exsolp3} and Theorem \ref{unique} we have the following
\begin{thm}\label{trunique}
Let $p=3$. Then we have the following
\begin{itemize}
\item[1.] The translation-invariant generalized $3$-adic Gibbs measures for HC model on the Cayley tree of order two exist iff one of the following statements hold:
   \begin{itemize}
    \item[1.a)] $|b|_3^2<27$,
    \item[1.b)] $|b|_3>27, 3|{\rm ord}_3b, b^*\in\mathbb Z_3^*[1,0]\cup\mathbb Z_3^*[2,2]$.
    \end{itemize}
 \item[2.] If the translation-invariant generalized $3$-adic Gibbs measures for HC model on the Cayley tree of order two exist, it will be unique.
\end{itemize}
\end{thm}

\subsection{Two-periodic GpGM}
By results of Section \ref{s32} we get the following theorems related to two-periodic GpGMs. 

\begin{thm}\label{Tp}Let $p\geq3$ be a prime and $N_{PGpGM}$ is number of $G_2^{(2)}$-perodic generalized $p$-adic Gibbs measures for HC model on the Cayley tree of order two. Then the following statements hold:
\begin{equation}\label{no}
N_{PGpGM}=\left\{
\begin{array}{lllllll}
2, \mbox ~{if}~ \left|\theta\right|_p>1;\\[2mm]
2, \mbox ~{if}~ \left|\theta\right|_p<1, 2\mid {\rm ord}_p \theta, \left(\frac{-\theta_0}{p}\right)=1;\\[2mm]
0, \mbox ~{if}~ \left|\theta\right|_p<1, 2\nmid {\rm ord}_p \theta \mbox ~{or}~\left(\frac{-\theta_0}{p}\right)=-1;\\[2mm]
2, \mbox ~{if}~ \left|\theta\right|_p=1, \left(\frac{\theta_0^2-4\theta_0}{p}\right)=1;\\[2mm]
0, \mbox ~{if}~ \left|\theta\right|_p=1, \left(\frac{\theta_0^2-4\theta_0}{p}\right)=-1;\\[2mm]
2, \mbox ~{if}~ \left|\theta\right|_p=1, 2\mid {\rm ord}_p (\theta-4), \left(\frac{|\theta-4|_p(\theta-4)(\operatorname{mod }p)}{p}\right)=1;\\[2mm]
0, \mbox ~{if}~ \left|\theta\right|_p=1, 2\nmid {\rm ord}_p (\theta-4)  \mbox ~{or}~ \left(\frac{|\theta-4|_p(\theta-4)(\operatorname{mod }p)}{p}\right)=-1;\\[2mm]
\end{array}\right.
\end{equation}
\end{thm}

\begin{thm}\label{Tp2} Let $p=2$. Then the following statements hold true:
\begin{equation}\label{no2}
N_{PGpGM}=\left\{
\begin{array}{lllllllllllllll}
2, \mbox ~{if}~ |\theta|_2>1;\\[2mm]
2, \mbox ~{if}~ |\theta|_2=1 \mbox~{and}~ \theta_2=1;\\[2mm]
0, \mbox ~{if}~ |\theta|_2=1 \mbox~{and}~ \theta_2=0;\\[2mm]
2, \mbox ~{if}~ |\theta|_2=\frac{1}{2} \mbox~{and}~\theta_1=0;\\[2mm]
0, \mbox ~{if}~ |\theta|_2=\frac{1}{2} \mbox~{and}~\theta_1=1;\\[2mm]
2, \mbox ~{if}~ |\theta|_2=\frac{1}{4} \mbox~{and}~ \theta_1=0,\theta_2=\theta_3=1,\theta_4=0;\\[2mm]
0, \mbox ~{if}~ |\theta|_2=\frac{1}{4} \mbox~{and}~ (\theta_1-1)\theta_2\theta_3(\theta_4-1)=0;\\[2mm]
0, \mbox ~{if}~ |\theta|_2=\frac{1}{8};\\[2mm]
2, \mbox ~{if}~ |\theta|_2=\frac{1}{16} \mbox~{and}~ \theta_1=0,\theta_2=1;\\[2mm]
0, \mbox ~{if}~ |\theta|_2=\frac{1}{16} \mbox~{and}~ \theta_1=1 \mbox~or~\theta_2=0;\\[2mm]
0, \mbox ~{if}~ |\theta|_2\leq\frac{1}{32}, 2\nmid {\rm ord}_p \theta;\\[2mm]
2, \mbox ~{if}~ |\theta|_2\leq\frac{1}{32},2\mid {\rm ord}_p \theta, \theta_1=\theta_2=0;\\[2mm]
0, \mbox ~{if}~ |\theta|_2\leq\frac{1}{32},2\mid {\rm ord}_p \theta, \theta_1=1 \mbox ~{or}~\theta_2=1.
\end{array}\right.
\end{equation}
\end{thm}
\begin{rk}
	In \cite{KMR} for the real-valued case of the HC-model it is proven that there are 1 or 3 two-periodic GMs. But Theorem \ref{Tp2} shows that $p$-adic version of the model may have 0 or 2 two-periodic GpGMs. 
\end{rk}
\section*{Acknowledgements}
The work supported by the fundamental project (number: F-FA-2021-425)  of The Ministry of Innovative Development of the Republic of Uzbekistan.

\end{document}